\documentclass[10pt,a4paper]{article}
\usepackage{amscd}
\usepackage{amsmath, upgreek}
\usepackage{graphicx}
\usepackage{amsfonts}
\usepackage{amssymb, wasysym}
\usepackage{amssymb}
\newtheorem{theorem}{Theorem}[section]
\newtheorem{lemma}[theorem]{Lemma}
\newtheorem{corollary}[theorem]{Corollary}
\newtheorem{proposition}[theorem]{Proposition}

\newtheorem{remark}[theorem]{Remark}

\numberwithin{equation}{section}

\begin{document}

\title{On the Schr\"{o}dinger equations with isotropic and anisotropic fourth-order dispersion}
\author{{Carlos Banquet} \\
{\small Universidad de C\'{o}rdoba, Departamento de Matem\'{a}ticas y Estad\'{\i}stica}\\
{\small A.A. 354, Monter\'{\i}a, Colombia.}\\
{\small \texttt{E-mail:cbanquet@correo.unicordoba.edu.co}}\vspace{.5cm}\\
%{{Lucas C. F. Ferreira} {\thanks{L.C.F. Ferreira was supported by 2010/19098-2 and CNPQ 305542/2009-5, Brazil.}}}\\
%{\small Universidade Estadual de Campinas, IMECC - Departamento de Matem\'{a}tica,} \\
%{\small {Rua S\'{e}rgio Buarque de Holanda, 651, CEP 13083-859, Campinas-SP, Brazil.}}\\
%{\small \texttt{E-mail:lcff@ime.unicamp.br}} \vspace{.5cm}\\
{{Elder J. Villamizar-Roa}{\thanks{Corresponding author.}}}\\
{\small Universidad Industrial de Santander, Escuela de Matem\'{a}ticas}\\
{\small A.A. 678, Bucaramanga, Colombia.} \\
{\small \texttt{E-mail:jvillami@uis.edu.co}}}
\date{\today}
\maketitle
%%%%%%%%%%%%%%%%%%%%%%%%%%% ABSTRACT %%%%%%%%%%%%%%%%%%%%%%%%%%%%%%%%%%%%%%%%%%%%%%%%%%%%
\begin{abstract}
This paper deals with the Cauchy problem associated to the nonlinear
fourth-order Schr\"{o}dinger equation with isotropic and anisotropic
mixed dispersion. This model is given by the equation $i\partial
_{t}u+\epsilon \Delta u+\delta A u+\lambda|u|^\alpha u=0,$ $x\in
\mathbb{R}^{n},$ $t\in \mathbb{R},$ where $A$ represents either the
operator $\Delta^2$ (isotropic dispersion) or
$\sum_{i=1}^d\partial_{x_ix_ix_ix_i},\ 1\leq d<n$ (anisotropic
dispersion), and $\alpha, \epsilon, \lambda$ are given real
parameters. We obtain local and global well-posedness results in
spaces of initial data with low regularity, based on weak-$L^p$
spaces. Our analysis also includes the Biharmonic and anisotropic
Biharmonic equation $(\epsilon=0);$ in that case, we obtain the existence of
self-similar solutions due their scaling
invariance property. In a second part, we analyze the convergence of the solutions for the nonlinear
fourth-order Schr\"{o}dinger equation $i\partial _{t}u+\epsilon
\Delta u+\delta \Delta^2 u+\lambda|u|^\alpha u=0$, as $\epsilon$
goes to zero, in $H^2$-norm, to the solutions of the corresponding
Biharmonic equation $i\partial _{t}u+\delta \Delta^2
u+\lambda|u|^\alpha u=0$.
\bigskip


\end{abstract}

\pagestyle{myheadings} \markright{Fourth-order Schr\"{o}dinger equation}
%%%%%%%%%%%%%%%%%%%%%%%%%%%%%%%%%%%% INTRODUCTION %%%%%%%%%%%%%%%%%%%%%%%%%%%%%%%%%%%%%%%
\section{Introduction}
This paper is devoted to the study of the Cauchy problem associated to  the following fourth-order Schr\"{o}dinger equation 
in $\mathbb{R}^n\times \mathbb{R}:$
\begin{equation}\label{FoSch}
\left\{
\begin{array}{lc}
i\partial _{t}u+\epsilon \Delta u+\delta A u+f(|u|)u=0, & x\in \mathbb{R}^{n},\ \ t\in \mathbb{R}, \\
u(x,0)=u_{0}(x), & x\in \mathbb{R}^{n}, \\
\end{array}
\right.
\end{equation}
where the unknown $u(x,t)$ is a complex-valued function in space-time $
\mathbb{R}^n\times \mathbb{R}, n\geq 1,$  $u_0$ denotes the initial
data and  $\epsilon$, $\delta$, are real parameters. The operator $A$ is defined by
\begin{equation}\label{DefOpe}
A u=\left\{
\begin{array}{lc}
\Delta^2u=\Delta\Delta u,\ \mbox{(isotropic dispersion)}, \\
\sum\limits_{i=1}^du_{x_ix_ix_ix_i},\ 1\leq d<n,\ \mbox{(anisotropic dispersion)}. \\
\end{array}
\right.
\end{equation}
The nonlinear term is given by $f(|u|)u$ where
 $f:\mathbb{R}\rightarrow \mathbb{R}$ satisfies
\[|f(x)-f(y)|\leq C_f|x-y|(|x|^{\alpha-1}+|y|^{\alpha-1}),\]
for some $1\leq \alpha<\infty$, $f(0)=0,$ and the constant $C_f >0$
is independent of $x,y \in \mathbb{R}.$ A typical case of a function
$f$ is given by $f(x)=\vert x\vert^\alpha.$

The class of fourth-order Schr\"{o}dinger equations has been widely
used in many branches of  applied science such as nonlinear
optics, deep water wave dynamics, plasma physics, superconductivity,
quantum mechanics and so on \cite{AceAngTur, Dysthe, Hirota, Ivano,
Kar, KarSha, WenFan}. If we consider $\epsilon=0$ in (\ref{FoSch}),
the resulting equation is the fourth-order nonlinear Schr\"{o}dinger
equation
\begin{eqnarray}\label{FoNLS}
i\partial _{t}u+\delta A u+f(|u|)u=0.
\end{eqnarray}
In particular, if we take $A=\Delta^2$ in (\ref{FoNLS}) we obtain the well-known
Biharmonic equation (BNLS)
\begin{eqnarray}\label{BNLS}
i\partial _{t}u+\delta \Delta^2 u+f(|u|)u=0,
\end{eqnarray}
introduced by Karpman \cite{Kar}, and Karpman and Shagalov
\cite{KarSha} to take into account the role played by the higher
fourth-order dispersion
 terms in formation and propagation of intense laser beams in a bulk medium with Kerr nonlinearity  \cite{Ivano}.
 Historically, (\ref{BNLS}) has been extensively studied in Sobolev spaces, see for instance
\cite{Fibich,Miao,Miao1,Pau2,Pau,Pau1,Segata,Wang,ZhuYanZha} and references
therein. Fibich \textit{et al} \cite{Fibich} established sufficient conditions for the global existence of solutions to the BNLS equation, for $\delta<0$  and $\delta>0$,
with initial data in $H^2(\Omega)$ being $\Omega$ a smooth bounded
domain of $ \mathbb{R}^n.$ Global existence and scattering theory
for the defocusing BNLS, in $H^2(\mathbb{R}^n),$ was established in
Pausander \cite{Pau2, Pau}.
 Wang in \cite{Wang} showed the global existence of
solutions  and a scattering result for the  BNLS (with a
nonlinearity of the form $|u|^pu$) with
 small initial radial data in the homogeneous Sobolev space ${\dot{H}^{s_c}(\mathbb{R}^n)}$ and
 dimensions $n\geq 2.$ Here $s_c=\frac{n}{2}-\frac{4}{p}$ and $s_c>-\frac{3n-2}{2n+1}$.
 The main ingredient of \cite{Wang} is the improvement of  the Strichartz estimatives associated to the BNLS for radial initial data; see also Zhu,
 Yang and Zhang \cite {ZhuYanZha}, where some results on blow-up solitons for the BNLS
 equation are established. More recently, Guo in \cite{Guo6} analyzed the global existence of solutions in Sobolev spaces and the asymptotic behavior
for the Cauchy problem associated to BNLS equation with combined
power-type nonlinearities. Finally, we recall a recent result of
Miao \textit{et al.} \cite{Miao} about the defocusing
energy-critical nonlinear BNLS equation  $iu_t
+ \Delta^2u = -|u|^{\frac{8}{d-4}}u,$ which establishes that any
finite energy solution is global and scatters both forward and
backward in time for dimensions $d\geq 9$.

When  $\epsilon\neq 0$ and $A$ is the biharmonic operator, equation (\ref{FoSch}) corresponds to the
following nonlinear Schr\"{o}dinger equation with isotropic
mixed-dispersion (INLS):
\begin{eqnarray}\label{INLS}
i\partial _{t}u+\epsilon \Delta u+\delta\Delta^2 u+f(|u|)u=0.
\end{eqnarray}
This equation was also introduced by Karpman \cite{Kar}, and Karpman
and Shagalov \cite{KarSha}, and it has been used as a model to
investigate the role played by the higher-order dispersion terms, in
formation and propagation of solitary waves in magnetic materials
where the effective quasi-particle mass becomes infinite. From the
mathematical point of view, the INLS equation has been studied
extensively in Sobolev and Besov spaces,
 see for instance \cite{GuoCui4,GuoCui3,Guo,GuoCui5,Fibich} and some references therein.
Fibich {\it et al.} \cite{Fibich} investigated the global existence of
solutions of (\ref{INLS}) in the class $C(\mathbb{R};
H^2(\mathbb{R}^n))$ by using the conservation laws. Moreover, the dynamic of the solutions and numerical simulations were also analyzed. These results were improved by Guo and Cui in
\cite{GuoCui4}. Local well-posedness of the Cauchy problem
(\ref{INLS}) in Sobolev spaces $H^s(\mathbb{R}^n),$ with $f(u)=\vert
u \vert^\alpha,$ $\frac{\alpha}{2}\geq\frac{4}{n},$
$s>s_0:=\frac{n}{2}-\frac{4}{\alpha},$ was obtained by Cui and
Guo in \cite{GuoCui5}. Additionally,  by using such
local result and the conservation laws, a global well-posedness results
in $H^2(\mathbb{R}^n)$ was also established. In \cite{GuoCui3} the authors proved some
results of local and global well-posedness on Besov spaces for
 dimensions $1\leq n\leq 4$; more exactly, the authors proved that
 the Cauchy problem associated to (\ref{INLS}), with $f(u)=\vert u \vert^\alpha,$ is local well possed in
 $C([-T,T];\dot{B}^{s_\alpha}_{2,q}(\mathbb{R}^n))$ and
 $C([-T,T];B_{2,q}^s(\mathbb{R}^n))$ for some $T>0,$ where
 $s_\alpha=\frac{n}{2}-\frac{4}{\alpha},$ $s>s_\alpha,$ $1\leq q \leq \infty.$
With respect to the global well-posedness in Sobolev space,  Guo in \cite{Guo}, considering
$f(u)=\vert u \vert^{2m},$ and using the I-method, proved the existence of global solutions in $H^s(\mathbb{R}^n)$ for $s>1+\tfrac{mn-9+\sqrt{(4m-mn+7)^2+16}}{4m},$ $4<mn<4m+2.$

Another important model considerated  in (\ref{FoSch}) is given by the case of anisotropic dispersion (ANLS), that is,
\begin{equation}\label{ANLS}
\begin{array}{lc}
i\partial _{t}u+\epsilon \Delta u+\delta\sum\limits_{i=1}^du_{x_ix_ix_ix_i}
+f(|u|)u=0.
\end{array}
\end{equation}
This model appears in the propagation of ultrashort laser pulses in a planar waveguide medium with
anomalous time-dispersion, and the propagation of solitons in fiber arrays (see Wen and Fan \cite{WenFan}
and Acevedes {\it et al.} \cite{AceAngTur}).  Results of local and global well-posedness for initial data in $H^s$-spaces were given in \cite{GuoCui5} and \cite{ZhaGuoSheWei}

In this paper we are interested in the local and global well-posedness of the general fourth-order Schr\"{o}dinger equation outside the framework of finite energy $H^s$-spaces. More exactly, we analyze the existence of local and global solutions for the Cauchy problem (\ref{FoSch}) in a new class of initial data based on 
weak-$L^p$ spaces (see definition below). Weak-$L^p$ spaces, also denoted by $L^{(p,\infty)},$ are natural extensions of Lebesgue
spaces $L^p$, in view of the Chebyshev inequality \cite{BL}. They contain singular functions with infinite $L^2$-mass such as homogeneous functions of degree $-\frac{n}{p}.$ Making a comparison between weak-$L^p$ spaces and $H^{s,l}$-spaces, it is known that the continuous
inclusion $H^{s,l}(\mathbb{R}^{n})\subset L^{(p,\infty)}(\mathbb{R}^{n})$
holds true for $s\geq0$ and $\frac{1}{p}\geq\frac{1}{l}-\frac{s}{n}$, and
$H^{s,l}$-spaces do not contain any weak-$L^{p}$ spaces if $s\in\mathbb{R},$
$1\leq l\leq2$ and $l\leq p$. In particular, $L^{(p,\infty)}(\mathbb{R}%
^{n})\not \subset H^{s,2}(\mathbb{R}^{n})=H^{s}(\mathbb{R}^{n})$ for all
$s\in\mathbb{R},$ when $p\geq2.$ However, $L^{(p,\infty)}\subset L^2_{loc}$ for $p>2.$ On the other hand, comparing equations (\ref{FoNLS}) with (\ref{INLS}) and (\ref{ANLS}),
we observe that equation (\ref{FoNLS}), with $f(\vert u \vert)=\vert
u\vert^\alpha$, unlike equations (\ref{INLS}) and (\ref{ANLS}), is invariant under the group of transformations
$u(x,t)\rightarrow u_\lambda(x,t),$ where
$u_\lambda(x,t)=\lambda^{\frac{4}{\alpha}}u(\lambda x,\lambda^4t)$,
$\lambda>0.$ Solutions which are invariant under the transformation
$u\rightarrow u_\lambda$ are called self-similar solutions. 
As pointed out in Dudley {\it et al.} \cite{Dudley} (see also \cite{LucEld1}), self-similarity type properties appear in a wide range of physical situations and they reproduce the structure of a phenomena in different spatio-temporal scales. A universal law governing self-similar scale invariance reveals the existence of internal symmetry and structure in a system. Thus, self-similar solutions naturally provide such a law for system (\ref{FoNLS}). In ultrafast nonlinear optics, self-similar dynamics have attracted a lot of interest and constitute an increasing field of research (see  \cite{Dudley} and references therein). For instance, in Fermann {\it et al.} \cite{Fermann} was showed that a type of self-similar parabolic pulse is an asymptotic solution to a nonlinear Schr\"{o}dinger equation with gain. In
order to obtain self-similar solutions we need to consider a norm
$\Vert \cdot\Vert$ defined on a space of initial data $u_0,$ which
is invariant with respect to the group of transformations
$u\rightarrow u_\lambda,$ that is, $\Vert {u_0}_\lambda\Vert=\Vert
{u_0}\Vert$ for all $\lambda>0;$ therefore $u_0$ must be a
homogeneous function of degree $-\frac{4}{\alpha}.$ However, $H^s$-spaces are not well adapted for studying this kind of solutions.
This fact represents an additional motivation to study the existence of global solutions of
Cauchy problem associated to (\ref{FoNLS}) with initial data outside  $H^s$-spaces, by using norms based on $L^{(p,\infty)}$. As consequence, the existence of forward
self-similar solutions for (\ref{FoNLS}) is obtained by assuming
$u_0$ a sufficiently small homogeneous function of degree
$-\frac{4}{\lambda}$. Because equation appearing in (\ref{FoSch}) does not
verify any scaling symmetries (in particular equations (\ref{INLS}) and (\ref{ANLS})), it is not likely to possess self-similar solutions. However, by using time decay estimates for the
respective fourth-order Schr\"{o}dinger group in weak-$L^p$ spaces,
we are able to obtain a result of existence of global solutions for
the Cauchy problem (\ref{FoSch}) in a class of function spaces
generated by the scaling of the Biharmonic equation (\ref{BNLS})
with $f(\vert u \vert)=\vert u\vert^\alpha$. In relation to the
existence of local in time solutions for (\ref{FoSch}) and in
particular, the Cauchy problem associated to the equation
(\ref{FoNLS}), we will prove a result of existence and uniqueness
for a large class of singular initial data, which includes
homogeneous functions of degree $-\frac n p$ for adequate values of $p.$
 The solutions obtained here can be physically interesting because, as was said, elements of $L^{(p,\infty)}$ have local finite $L^2$-mass (that is, they belong to $L^2_{loc}$), for $p > 2$.  In addition, for initial data in $H^s(\mathbb{R}^n)$, the corresponding solution belongs to $H^s(\mathbb{R}^n),$ which shows that the constructed data-solution map in $L^{(p,\infty)}$ recovers the $H^s$-regularity and it is compatible with the $H^s$-theory. 

It is worthwhile to remark that the existence of local and global
solutions for dispersive equations with initial data outside the
context of finite $L^2$-mass, such as weak-$L^r$ spaces, has been
analyzed for the classical Schr\"{o}dinger equation, coupled
Schr\"{o}dinger equations, Davey-Stewartson system, which are
models characterized by having scaling relation (cf. \cite{WC, LucEld1, LucEldPab,  EldJean}). Existence of solutions in the
framework of weak-$L^r$ spaces for models which have no scaling
relation, have been explored in the case of Boussinesq and
Schor\"{o}dinger-Boussinesq system in \cite{CaLuEld, Fer-Bousq} and more recently, in the context of Klein-Gordon-Schr\"{o}dinger system \cite{CaLuEld1}.

In order to state our results, we start by establishing the definition of mild solution for the Cauchy problem (\ref{FoSch}). A mild solution for (\ref{FoSch}) is a function  $u$ satisfying the following integral equation
\begin{equation}\label{IntEqu}
u(x,t)=G_{\epsilon,\delta}(t)u_0(x)+i\int_0^tG_{\epsilon,\delta}(t-\tau)f(|u(x,\tau)|)u(x,\tau)d\tau,
\end{equation}
where  $G_{\epsilon,\delta}(t)$ is the free group associated to the
linear Fourth-order Schr\"{o}dinger equation, that is,
\begin{equation}\label{DefGe}
G_{\epsilon,\delta}(t)\varphi=\left\{
\begin{array}{lll}
J_{\epsilon,\delta}(\cdot,t)\ast\varphi, \ \ \mbox{if}\ A=\Delta^2,\\
I_{\epsilon,\delta}(\cdot,t)\ast\varphi,\ \ \mbox{if}\
A=\sum\limits_{i=1}^d\partial_{x_ix_ix_ix_i},
\end{array}
\right.
\end{equation}
 for all $\varphi\in \mathcal{S}'(\mathbb{R}^n), $ where
\begin{align*}
J_{\epsilon,\delta}(x,t)&=(2\pi)^{-n}\int_{\mathbb{R}^n}e^{ix \xi -it\left(\epsilon|\xi|^2-\delta|\xi|^4\right)}d\xi\\
I_{\epsilon,\delta}(x,t)&=\left((2\pi)^{-d}\prod_{j=1}^d
\int_{\mathbb{R}}e^{ix_j\xi_j-it(\epsilon\xi_j^2-\delta\xi_j^4)}d\xi_j\right)\times \left((2\pi)^{-(n-d)}\prod_{j=d+1}^n\int_{\mathbb{R}}e^{ix_j\xi_j-it\epsilon\xi_j^2}d\xi_j\right)\\
&\equiv I^1_{\epsilon,\delta}(x,t)I^2_{\epsilon,\delta}(x,t).
\end{align*}

Before to precise our results, briefly we recall some notation and
facts about Lorentz spaces, see Bergh and L\"{o}fstr\"{o}m \cite{BL}, which will be our scenario to
establish existence results. Lorentz spaces $L^{(p,d)}$ are
defined as the set of measurable function $g$ on $\mathbb{R}^{n}$
such that the quantity
\begin{equation*}
\Vert g\Vert _{(p,d)}=\left\{
\begin{array}{lll}
\left( \frac{p}{d}\displaystyle\int_{0}^{\infty }\left[
t^{\frac{1}{p}}g^{\ast \ast }(t)\right] ^{d}\frac{dt}{t}\right)
^{\frac{1}{d}}, &
\mbox{if} & 1<p<\infty \mbox{, }1\leq d<\infty, \\
\sup_{t>0}t^{\frac{1}{p}}g^{\ast \ast }(t), & \mbox{if} & 1<p\leq
\infty \mbox{, }d=\infty ,
\end{array}
\right.
\end{equation*}
is finite. Here $g^{\ast \ast }(t)=\frac{1}{t}\int_{0}^{t}g^{\ast }(s)\mbox{ }ds$ and
\begin{equation*}
 g^{\ast}(t)=\inf \{s>0:\mu \left( \{{x}\in \Omega :|g({x})|>s\}\right) \leq t\},\ \ t>0,
\end{equation*}
with $\mu $ denoting the Lebesgue measure. In particular,
$L^{p}(\Omega)=L^{(p,p)}(\Omega )$ and, when $d=\infty ,\
L^{(p,\infty )}(\Omega )$ are called weak-$L^{p}$ spaces.
Furthermore, $L^{(p,d_{1})}\subset L^{p}\subset L^{(p,d_{2})}\subset
L^{(p,\infty )}$ for $1\leq d_{1}\leq p\leq d_{2}\leq \infty $. In
particular, weak-$L^{p}$ spaces contain singular functions with
infinite $L^{2}$-mass such as homogeneous functions of degree
$-\frac n p$. Comparing to $H^{s,l}$-spaces, the continuous inclusion
$H^{s,l}(\mathbb{R} ^{n})\subset L^{(p,\infty )}(\mathbb{R}^{n})$
holds true for $s\geq 0$ and $ \frac{1}{l}-\frac{s}{n}\leq
\frac{1}{p}$, and $H^{s,l}$-spaces do not contain any weak-$L^{p}$
spaces if $s\in \mathbb{R},$ $1\leq l\leq 2$ and $ l\leq p$. In
particular, $L^{(p,\infty )}(\mathbb{R}^{n})\not\subset
H^{s,2}(\mathbb{R}^{n})=H^{s}(\mathbb{R}^{n})$ for all $s\in
\mathbb{R},$ when $ p\geq 2.$  Finally, a helpful fact about Lorentz
spaces is the validity of the H\"{o}lder inequality, which reads
\begin{equation*}
\Vert gh\Vert _{(r,s)}\leq C(r)\Vert g\Vert _{(p_{1},d_{1})}\Vert
h\Vert_{(p_{2},d_{2})},
\end{equation*}
for $1<p_{1}\leq \infty ,$ $1<p_{2},$ $r<\infty ,$
$\frac{1}{p_{1}}+\frac{1}{p_{2}}<1,$
$\frac{1}{r}=\frac{1}{p_{1}}+\frac{1}{p_{2}}$, and $s\geq 1$ satisfies $\frac{1}{d_{1}}+\frac{1}{d_{2}}\geq \frac{1}{s}$.\\

In this paper we obtain new results of local and global existence on Schr\"{o}dinger equations with isotropic
and anisotropic fourth-order dispersion. First, we prove the existence of local-in-time solutions to the
integral equation (\ref{IntEqu}) (see Theorem
\ref{LocalTheo}).
For the existence of local solutions, fixed $0<T<\infty,$ we
consider the space $\mathcal{G}_{\beta}^T$ of Bochner measurable
functions $u:(-T,T)\rightarrow L^{(p(\alpha+1),\infty)}$ such that
\begin{equation*}
\|u\|_{\mathcal{G}_{\beta}^T}=\sup_{-T<t<T}|t|^{\beta}\|u(t)\|_{(p(\alpha+1),\infty)},
\end{equation*}
where
\begin{equation}\label{Defbeta}
\beta=\left\{
\begin{aligned}
&\frac{n\alpha}{4p(\alpha+1)},\ \mbox{if}\ A=\Delta^2,\\
&\frac{(2n-d)\alpha}{4p(\alpha+1)},\ \mbox{if}\
A=\sum_{i=1}^d\partial_{x_ix_ix_ix_i},
\end{aligned}
\right.
\end{equation}
and $p$ is such that the pair $(\frac{1}{p},\frac{1}{p(\alpha+1)})$
belongs to  the set $\Xi_0\setminus\partial \Xi_0$ where $\Xi_0$ is
the quadrilateral $R_0P_0BQ_0$, with $ B=(1,0), \ \ P_{0}=(2/3,0),\
\ Q_0=(1,1/3) \ \ \text{and}\ \ R_0=( 1/2,1/2). $ The exponent
$\beta$ in  (\ref{Defbeta}), and the restriction of $p,$ correspond
to the time decay of the group $G_{\epsilon, \delta}(t)$ on Lorentz spaces
(see Proposition \ref{LinEstLoc} below). The initial data is such
that $\|G_{\epsilon,\delta}(t)u_0\|_{\mathcal{G}_{\beta}^T}$ is
finite. As a consequence, some results of local existence in Sobolev spaces can be recovered (see Remark \ref{rem1}).\\

Second, we analyze the existence of
global-in-time solutions (see Theorem \ref{GlobalTheo}). For that we
define the space $\mathcal{G}^{\infty}_{\sigma}$ as the set of
Bochner measurable functions $u:(-\infty ,\infty )\rightarrow
L^{(\alpha+2,\infty )}$ such that
\[
\| u\|_{\mathcal{G}^{\infty}_{\sigma}}=\sup_{-\infty <t<\infty}|t|^\sigma\|
u(t)\| _{(\alpha+2,\infty )}<\infty,
\]
where $\sigma$ is  given by
\begin{equation}\label{Defsigma}
\sigma=\left\{
\begin{aligned}
&\frac{1}{\alpha}-\frac{n}{4(\alpha+2)},\ \mbox{if}\ A=\Delta^2,\\
&\frac{1}{\alpha}-\frac{2n-d}{4(\alpha+2)},\ \mbox{if}\
A=\sum_{i=1}^d\partial_{x_ix_ix_ix_i}.
\end{aligned}
\right.
\end{equation}
Observe that the value
$\sigma=\frac{1}{\alpha}-\frac{n}{4(\alpha+2)}$ in (\ref{Defsigma})
is the unique one such that the norm $\|
u\|_{\mathcal{G}^{\infty}_{\sigma}}$ becomes invariant by the
scaling of  Biharmonic equation with $f(u)=\vert u \vert^\alpha.$ In order
to obtain existence of global solutions,  we consider the following
class of initial data
\begin{eqnarray}\label{initial_data}
\mathcal{D}_\sigma\equiv\{\varphi\in \mathcal{S}'(\mathbb{R}^n):
\sup_{-\infty <t<\infty}t^\sigma\Vert
G_{\epsilon,\delta}(t)\varphi\Vert_{(\alpha+2,\infty)}<\infty\}.
\end{eqnarray}
As consequence, if we consider the Biharmonic or anisotropic Biharmonic equation, i.e., $\epsilon=0$ in (\ref{FoSch}), we obtain the
existence of self-similar solutions by assuming $u_0$ a sufficiently
small homogeneous function of degree $-\frac{4}{\alpha}$ (see Corollary \ref{self}).\\

As it was said, formally, when we drop the second order  dispersion
term in $i\partial _{t}u+\epsilon \Delta u+\delta \Delta^2
u+f(|u|)u=0,$ i.e.,  taking $\epsilon=0$, we obtain the Biharmonic
equation $i\partial _{t}u+\delta \Delta^2 u+f(|u|)u=0.$
 However, to the best of our
knowledge, the vanishing second order dispersion limit has not been
addressed. We observe that the analysis of vanishing dispersion
limits can be seen as an interesting issue in dispersive PDE theory,
because it permits to describe qualitative properties between
different models. We recall, for instance, that in fluid mechanics,
the vanishing viscosity limit of the incompressible Navier-Stokes
equations is a classical issue \cite{FPV,Iftimie}. This is the
motivation of the second aim of this paper. We study the convergence as $\epsilon$
goes to zero, in the $H^2$-norm, of the solution of Cauchy problem
(\ref{FoSch}), with $A=\Delta^2,$ to the corresponding Cauchy
problem associated to biharmonic equation (\ref{BNLS}). In the the anisotropic case, i.e.,
$A=\sum_{i=1}^d\partial_{x_ix_ix_ix_i},$  the vanishing second order dispersion limit is not clear, because we are not able to bound $\Vert \nabla u_\epsilon\Vert_{L^2}$  or $ \Vert u_\epsilon\Vert^2_{H^1}+\sum_{i=1}^d\Vert  u_{\epsilon_{x_ix_i}}\Vert^2_{L^2}$ in terms of the conserved quantities associated  to (\ref{FoSch}) and  independently of $\epsilon.$  This is an interesting question to be considered as future research.

The rest of this paper is organized as follows. In Section 2 we
establish some linear and nonlinear estimates which are fundamental
in order to obtain our results of local and global mild solutions. 
In Section 3 we state and prove our results of local and global
solutions. Finally, in Section 4, we give a result about vanishing second order
dispersion limit.
%Without loss of
%generally we will consider $ \delta=\pm1; \varepsilon=\pm1,0.$
%%%%%%%%%%%%%%%%%%%%%%%%%%%%% LINEAR AND NONLINEAR ESTIMATES  %%%%%%%%%%%%%%%%%%%%%%%%%%%%%%%%%%%
\section{Linear and nonlinear estimates}
In this section we
establish some linear and nonlinear estimates which are fundamental
in order to obtain our results of local and global mild solutions.
We start by rewriting Theorem 2, Section 3, of Cui \cite{Cui1} for the
case $n = 1$ and Theorem 2, Section 3, of Cui \cite{Cui2}
for the case $n \geq 2$ (see also Lemma 2.1 in Guo and Cui
\cite{GuoCui, GuoCui2}). For this purpose we denote $\Xi_0$ the quadrilateral
$R_0P_0BQ_0$ in the $(1/p,1/q)$ plane, where
\begin{equation*}
B=(1,0), \ \ P_{0}=(2/3,0),\ \ Q_0=(1,1/3) \ \ \text{and}\ \ R_0=( 1/2,1/2).
\end{equation*}
$\Xi_0$ comprises the apices $B, R_0$ and all the edges $BP_0,$
$BQ_0,$ $P_0R_0$ and $Q_0R_0,$ but does not comprise the apices
$P_0$ and $Q_0.$

%%%%%%%%%%%%%%%% PROPOSITION %%%%%%%%%%%%%%%%%%%%%%%%%%%%%%%%%%%%%%%%%%%%%%%%%%%%%
\begin{proposition}\label{LemmaCui} Given  $T>0$ and a pair of positive numbers $(p,q)$ satisfying $(1/p,1/q)\in\Xi_0,$ there exists a
 constant $C=C(T,p,q)>0$ such that for any $\varphi\in L^p(\mathbb{R}^n)$ and $-T\leq t\leq T$ it holds
\begin{eqnarray*}
\Vert G_{\epsilon,\delta}(t)\varphi\Vert _{L^q}\leq C\vert
t\vert^{-b_l}\Vert \varphi \Vert _{L^{p}},
\end{eqnarray*}
where
\begin{equation}\label{e1}
b_l=\left\{
\begin{array}{lll}
\frac{n}{4}\left(\frac{1}{p}-\frac{1}{q}\right),\ \mbox{if}\ A=\Delta^2,\\
\frac{2n-d}{4}\left(\frac{1}{p}-\frac{1}{q}\right),\ \mbox{if}\
A=\sum\limits_{i=1}^d\partial_{x_ix_ix_ix_i}.
\end{array}
\right.
\end{equation}

 Moreover, if $\epsilon=0$ the above estimate holds for all
$t\neq 0.$
\end{proposition}

The last inequality is not convenient to obtain a result of global
well-posedness  because the constant $C$ depends on $T.$ In order to
overcome this problem we establish a different result which follows
from a standard scaling argument.
%%%%%%%%%%%%%%%%%%%%%%% LEMMA SIN DEPENDENCIA DE T%%%%%%%%%%%%%%%%%%%%%%%%%%%
\begin{lemma}\label{LemmaNoT}
If $\frac 1 p+\frac 1{p'}=1$ with $p\in[1,2],$ then there exists a constant $C$
independent of $\epsilon,\delta$ and $t$ such that
\[\|G_{\epsilon,\delta}(t)\varphi\|_{L^{p'}}\leq C |t|^{-b_g}\|\varphi\|_{L^p},\ \varphi\in L^p(\mathbb{R}^n),\]
for all $t\neq 0,$ where
\begin{equation}\label{e2}
b_g=\left\{
\begin{array}{lll}
\frac{n}{4}\left(\frac{2}{p}-1\right),\ \mbox{if}\ A=\Delta^2,\\
\frac{2n-d}{4}\left(\frac{2}{p}-1\right),\ \mbox{if}\
A=\sum\limits_{i=1}^d\partial_{x_ix_ix_ix_i}.
\end{array}
\right.
\end{equation}
\end{lemma}
\begin{proof} It is clear that
$\|G_{\epsilon,\delta}(t)\varphi\|_{L^2}=\|\varphi\|_{L^2},$ in both
cases, the isotropic and anisotropic dispersion.  Now,  for the
isotropic case define $h(\xi):=\frac{z\xi}{t}-(\epsilon
\xi^2-\delta\xi^4)$ and since $|h^{(4)}(\xi)|=24,$ we can use
Proposition VIII. 2 in Stein \cite{SteinLibro1}  in order to obtain
\[\left |\int_{-\infty}^{\infty}e^{ith(\xi)}d\xi\right |\leq C|t|^{-1/4}.\]
Note that the constant $C$ given above does not depend on
$\epsilon$ and $\delta.$ From Young inequality we have that
$\|G_{\epsilon,\delta}(t)\varphi\|_{\infty}\leq C
|t|^{-1/4}\|\varphi\|_{L^1}.$ Then the result follows by real
interpolation.\newline The anisotropic case is obtained in a similar
way. Indeed, we only need to note that
\[|I^1_{\epsilon,\delta}(x,t)|\leq C_1|t|^{-\frac{d}{4}}\ \ \ \text{and}\ \ \ |I^1_{\epsilon,\delta}(x,t)|\leq C_2|t|^{-\frac{n-d}{2}},\]
where $C_1$ and $C_2$ are independent of $t,
\epsilon$ and $\delta.$  Consequently
\[|I_{\epsilon,\delta}(x,t)|=|I^1_{\epsilon,\delta}(x,t)I^2_{\epsilon,\delta}(x,t)|\leq C|t|^{-\frac{2n-d}{4}}.\]
Then, the proof of lemma is finished.
\end{proof}
%%%%%%%%%%%%%%%%%%%%%%%%%LINEAR ESTIMATE LOCAL %%%%%%%%%%%%%%%%%%%%%%%%%%%%%%%%%%%%
\begin{lemma}\label{LinEstLoc} Let $T>0,$ $1\leq d\leq \infty $  and  $1\leq p,q\leq \infty$ satisfying $(1/p,1/q)\in\Xi_0\setminus \partial \Xi_0.$ Then, there exists a positive
constant $C=C(T,p,q)>0$ such that
\begin{equation}
\Vert G_{\epsilon,\delta}(t)\varphi\Vert _{(q,d)}\leq C\vert
t\vert^{-b_l}\Vert \varphi\Vert _{(p,d)},
\end{equation}
for all  $-T\leq t\leq T$ and  measurable $\varphi.$ Here, $b_l$ is
defined as in (\ref{e1}). Moreover, if $\epsilon=0$ the above
estimate holds for all $t\neq 0.$
\end{lemma}

\begin{proof} We only make the proof of the isotropic case; the anisotropic case can be proved in an analogous way.
Since $\Xi_0$ is convex we can chose $(1/{p_0},1/{q_0}),$
$(1/{p_1},1/{q_1})\in \Xi_0$  such that $\frac{1}{p}=\frac{\theta
}{p_{0}}+\frac{1-\theta }{p_{1}}$ and $\frac{1}{q}=\frac{\theta
}{q_{0}}+\frac{1-\theta }{q_{1}},$ with $0<\theta <1.$ From
Proposition \ref{LemmaCui} we have that
$G_{\epsilon,\delta}(t):L^{p_0}\rightarrow L^{q_0}$ and
$G_{\epsilon,\delta}(t):L^{p_1}\rightarrow L^{q_1},$ with
norms bounded by
\[
\Vert G_{\epsilon,\delta}(t)\Vert _{p_0\rightarrow q_0}\leq C\vert t\vert ^{-n/4(1/{p_0}-1/{q_0})}\]
 and
 \[\Vert G_{\epsilon,\delta}(t)\Vert _{p_1\rightarrow q_1}\leq C\vert t\vert ^{-n/4(1/{p_1}-1/{q_1})}.\]
Since $L^{p}=L^{(p,p)},$ using real interpolation we get
\begin{align*}
\Vert G_{\epsilon,\delta}(t)\Vert _{(p,d)\rightarrow (q,d)}&\leq C\vert t\vert^{-n/4(1/{p_0}-1/{q_0})\theta}\vert t\vert ^{-n/4(1/{p_1}-1/{q_1})(1-\theta)}\\
&=C\vert t\vert ^{-n/4(1/p-1/q)},
\end{align*}
which finishes the proof of the lemma.
\end{proof}

In the same spirit of Lemma \ref{LinEstLoc} one can obtain the next result, which gives a linear estimate in Lorentz spaces. The proof follows from Lemma \ref{LemmaNoT} and real interpolation.
We omit it.
%%%%%%%%%%%%%%% LEMMA LINEAR ESTIMATE GLOBAL %%%%%%%%%%%%%%%%%%%%%%
\begin{lemma}\label{LinEstGlo} Let $1\leq d\leq \infty ,$ $1< p< 2$ and $p^{\prime }$ such that $\frac{1}{p}+\frac{1}{p^{\prime }}=1.$ Then, there exists a positive
constant $C$ such that
\begin{equation}
\Vert G_{\epsilon,\delta}(t)\varphi\Vert _{(p^{\prime },d)}\leq
C\vert t\vert^{-b_g}\Vert \varphi\Vert _{(p,d)},
\end{equation}
for all $t\neq 0$ and measurable $\varphi.$ Here $b_g$ is as
in (\ref{e2}).
\end{lemma}

%%%%%LEMMA TO LIMIT IN L^P WEAK SPACES%%%%%%%%%%%%%%%%

From now on we denote the nonlinear part of the integral equation (\ref{IntEqu}) by
\begin{equation*}
\mathcal{F}(u(x,t))=i\int_{0}^t
G_{\epsilon,\delta}(t-\tau)f(|u(x,\tau)|)u(x,\tau)d\tau.
\end{equation*}

In the next lemma we estimate the nonlinear term $\mathcal{F}(u)$ in
the norm $\Vert\cdot \Vert_{\mathcal{G}_{\sigma}^{\infty }},$ which
is crucial in order to obtain existence of global mild solutions.
\begin{lemma}\label{EstNonGlo}
Let $1\leq \alpha<\infty $ and  assume that $(\alpha +1)\sigma<1.$
Then,
\begin{enumerate}
\item If $\frac{n\alpha}{4(\alpha+2)}<1$ and $A=\Delta^2,$ then there exists a constant $C_1>0$ such that
\begin{align} \label{Global-estim}
&\notag \Vert \mathcal{F}(u)-\mathcal{F}(v)\Vert_{\mathcal{G}_{\sigma}^{\infty }} \leq \\
&C_1\sup_{-\infty<t<\infty}|t|^{\sigma}\Vert u-v\Vert _{(\alpha+2,\infty)}
\sup_{-\infty <t<\infty}|t|^{\alpha\sigma }\left[ \Vert u\Vert_{(\alpha+2,\infty)}^{\alpha}+\Vert v \Vert _{(\alpha+2,\infty )}^{\alpha}\right],
\end{align}
for all $u,v$ such that the right hand side of (\ref{Global-estim})
is finite.
\item If $\frac{(2n-d)\alpha}{4(\alpha+2)}<1$ and $A=\sum_{i=1}^d\partial_{x_ix_ix_ix_i},$ then there exists a constant $C_2>0$ such that
\begin{align} \label{Global-estim2}
&\notag \Vert \mathcal{F}(u)-\mathcal{F}(v)\Vert_{\mathcal{G}_{\sigma}^{\infty }} \leq \\
&C_2\sup_{-\infty<t<\infty}|t|^{\sigma}\Vert u-v\Vert _{(\alpha+2,\infty)}
\sup_{-\infty <t<\infty}|t|^{\alpha\sigma }\left[ \Vert u\Vert_{(\alpha+2,\infty)}^{\alpha}+\Vert v \Vert _{(\alpha+2,\infty )}^{\alpha}\right],
\end{align}
for all $u,v$ such that the right hand side of (\ref{Global-estim2})
is finite.
\end{enumerate}
\end{lemma}

\begin{proof}  Without loss of generality, we can consider only the case
$t>0$. Using Lemma  \ref{LinEstGlo}, the properties of $f$ and
H\"{o}lder inequality, we have
\begin{align*}
 \Vert \mathcal{F}(u)-\mathcal{F}(v)\Vert _{(p',\infty)}&\leq C\int_{0}^{t}(t-\tau)^{-\frac{n(2-p)}{4p}} \Vert f(|u|)u-f(|v|)v\Vert _{(p,\infty )}d\tau \\
&\leq C \int_{0}^{t}(t-\tau)^{-\frac{n(2-p)}{4p}} \Vert |u-v|(|u|^{\alpha}+|v|^{\alpha})\Vert _{(p,\infty )}d\tau\\
& \leq C\int_{0}^{t}(t-\tau)^{-\frac{n(2-p)}{4p}}\Vert
u-v\Vert_{(p',\infty )}\left[ \Vert u\Vert _{(p',\infty
)}^{\alpha}+\Vert v \Vert _{(p',\infty)}^{\alpha}\right] d\tau.
\end{align*}
Since $\frac{1}{p}+\frac{1}{p'}=1$ and we used the H\"{o}lder inequality,  we get the restriction $p'=\alpha+2.$ Hence
\begin{align*}
 &\Vert \mathcal{F}(u)-\mathcal{F}(v)\Vert _{(\alpha+2,\infty)}\leq \\
 &C\int_{0}^{t}(t-\tau)^{-\frac{n\alpha}{4(\alpha+2)}}\Vert u-v\Vert _{(\alpha+2,\infty )}\left[ \Vert u\Vert _{(\alpha+2,\infty )}^{\alpha}+\Vert v \Vert _{(\alpha+2,\infty)}^{\alpha}\right] d\tau\leq\\
&  C\sup_{t>0}t^{\sigma}\Vert
u-v\Vert_{(\alpha+2,\infty)}\sup_{t>0}t^{\alpha\sigma}\left[ \Vert
u\Vert_{(\alpha+2,\infty )}^{\alpha}+\Vert v\Vert_{(\alpha+2,\infty
)}^{\alpha}\right]t^{-\sigma}t^{1-\frac{n\alpha}{4(\alpha+2)}-\sigma\alpha}.
\end{align*}
From $1-\frac{n\alpha}{4(\alpha+2)}-\sigma\alpha=0,$ we
conclude that
\begin{align}\label{DesNon2aa}
 \notag t^{\sigma}&\Vert \mathcal{F}(u)-\mathcal{F}(v)\Vert _{(\alpha+2,\infty )} \leq \\
 &C\sup_{t>0}t^{\sigma}\Vert u-v\Vert _{(\alpha+2,\infty)} \sup_{t>0}t^{\alpha\sigma}\left[ \Vert u\Vert_{(\alpha+2,\infty )}^{\alpha}+\Vert v \Vert _{(\alpha+2,\infty )}^{\alpha} \right].
\end{align}
Taking the supremum in (\ref{DesNon2aa}) we conclude the proof of the estimate (\ref{Global-estim}). The proof of (\ref{Global-estim2}) follows in a similar way.
\end{proof}

%%%%%%%%%%%%%%%%%%%%%%%%%%%% LEMMA NONLINEAR ESTIMATES 2 %%%%%%%%%%%%%%%%%%%%%%%%%%%%%
In the next lemma we estimate the nonlinear term $\mathcal{F}(u)$ in
the norm $\Vert\cdot \Vert_{\mathcal{G}_{\beta}^T},$ which is
crucial in order to obtain existence of local-in-time mild solutions. Here we  use the notation $A\apprle B$ which means that there exists a constant $c>0$ such that $A\leq cB.$
\begin{lemma}\label{NonEstLoc}
Let $1\leq \alpha<\infty,$  and $(1/p,1/{(\alpha+1)p})\in\Xi_0\setminus\partial\Xi_0.$
\begin{enumerate}
\item If $\frac{n\alpha}{4p}<1$ and $A=\Delta^2,$ then there exists a constant $C_3>0$ such that
\begin{align}
\Vert \mathcal{F}(u)-&\mathcal{F}(v)\Vert_{\mathcal{G}_{\beta}^T} \leq  C_3\sup_{-T<t<T}|t|^{\beta}\Vert u-v\Vert _{((\alpha+1)p,\infty)}\nonumber\\
&\times\sup_{-T<t<T}|t|^{\beta\alpha}\left[ \Vert u\Vert_{((\alpha+1)p,\infty )}^{\alpha}+\Vert v\Vert _{((\alpha+1)p,\infty )}^{\alpha}\right] T^{1-\beta(\alpha+1)},\label{local-estim1}
\end{align}
for all $u,v$ such that the right hand side of (\ref{local-estim1}) is finite.
\item If $\frac{(2n-d)\alpha}{4p}<1$ and $A=\sum_{i=1}^d\partial_{x_ix_ix_ix_i},$ then there exists a constant $C_4>0$ such that
\begin{align}
 \Vert \mathcal{F}(u)-&\mathcal{F}(v)\Vert_{\mathcal{G}_{\beta}^T} \leq  C_4\sup_{-T<t<T}|t|^{\beta}\Vert u-v\Vert _{((\alpha+1)p,\infty)}\nonumber\\
 &\times \sup_{-T<t<T}|t|^{\beta\alpha}\left[ \Vert u\Vert_{((\alpha+1)p,\infty )}^{\alpha}+\Vert v\Vert _{((\alpha+1)p,\infty )}^{\alpha}\right] T^{1-\beta(\alpha+1)},\label{local-estim2}
\end{align}
for all $u,v$ such that the right hand side of (\ref{local-estim2}) is finite.
\end{enumerate}
\end{lemma}

\begin{proof} We only make the proof of the first inequality; the proof of the second one  is analogous. Without loss of generality suppose that $t>0.$ Then, from Lemma \ref{LinEstLoc}, the properties of $f$ and  H\"{o}lder inequality, we obtain
\begin{align*}
 \Vert \mathcal{F}(u)-\mathcal{F}(v)\Vert _{(q,\infty)}&\leq \int_{0}^{t}(t-\tau)^{-b_l}\left\Vert f(|u|)u-f(|v|)v \right\Vert _{(p,\infty )}d\tau\\
 & \leq C\int_{0}^{t}(t-\tau)^{-b_l}\left\Vert |u-v|(|u|^{\alpha}+|v|^{\alpha})\right\Vert _{(p,\infty )}d\tau\\
 & \leq C\int_{0}^{t}(t-\tau)^{-b_l}\Vert u-v\Vert_{(q,\infty)}\left(\Vert u\Vert^{\alpha}_{(q,\infty)}+\Vert v\Vert^{\alpha} _{(q,\infty )}\right)d\tau.
\end{align*}
Since we used the H\"{o}lder inequality the next restriction appears $q=(\alpha+1)p.$ Therefore
\begin{align*}
&\Vert \mathcal{F}(u)-\mathcal{F}(v)\Vert _{((\alpha+1)p,\infty)}\apprle\\
&\int_{0}^{t}(t-\tau)^{-\frac{n\alpha}{4p(\alpha+1)}}\Vert u-v\Vert_{((\alpha+1)p,\infty)}\left(\Vert u\Vert^{\alpha}_{((\alpha+1)p,\infty)}+\Vert v\Vert^{\alpha} _{((\alpha+1)p,\infty )}\right)d\tau\apprle \\
&\sup_{0<t<T}t^{\beta}\Vert u-v\Vert_{((\alpha+1)p,\infty)}\sup_{0<t<T}t^{\alpha\beta}\left[ \Vert u\Vert_{((\alpha+1)p,\infty )}^{\alpha}+\Vert v\Vert_{((\alpha+1)p,\infty )}^{\alpha}\right] t^{1-\beta(\alpha+2)}.
\end{align*}
Hence
\begin{align*}
t^{\beta} \Vert \mathcal{F}(u)-\mathcal{F}(v)\Vert& _{((\alpha+1)p,\infty)} \leq C\sup_{0<t<T}t^{\beta}\Vert u-v\Vert _{((\alpha+1)p,\infty)}\\
&\times \sup_{0<t<T}t^{\beta\alpha}\left[ \Vert u\Vert_{((\alpha+1)p,\infty )}^{\alpha}+\Vert v\Vert _{((\alpha+1)p,\infty )}^{\alpha}\right] T^{1-\beta(\alpha+1)}.
\end{align*}
Taking supremum on $t$ in the last inequality, we obtained the desired result.
\end{proof}
%%%%%%%%%%%%%%%%%%%% Main results%%%%%%%%%%%%%%%%%%%%%%%%%%%%%%%%%%%%%
\section{Local and global solutions}
In this section we prove some results of local and global well-posedness for the Schr\"{o}dinger equations with isotropic and anisotropic fourth-order dispersion in the setting of Lorentz spaces. 
%%%%%%%%%%%%%%%%%%%%%%%%% LOCAL IN TIME SOLUTIONS %%%%%%%%%%%%%%%%%%%%%%%%%%%%%%%%%%%%%%%%%%%%%
\subsection{Local-in-time solutions}
\begin{theorem}[Local-in-time solutions]
\label{LocalTheo} Let $1\leq \alpha<\infty,$ and
$(1/p,1/{(\alpha+1)p})\in\Xi_0\setminus\partial\Xi_0.$ Consider
$\frac{n\alpha}{4p}<1$ if $A=\Delta^2,$ or
$\frac{(2n-d)\alpha}{4p}<1$ if
$A=\sum_{i=1}^d\partial_{x_ix_ix_ix_i}.$ If $u_{0}\in
\mathcal{S}'(\mathbb{R}^n)$ such that
$\|G_{\epsilon,\delta}(t)u_0\|_{\mathcal{G}_{\beta}^T}$ is finite,
then there exists $0<T^{\ast }\leq T<\infty $ such that the initial
value problem (\ref{FoSch}) has a mild solution $u\in
\mathcal{G}_{\beta}^{T^*},$ satisfying $u(t)\rightharpoonup u_0$ in
$\mathcal{S}'(\mathbb{R}^n)$ as $t\rightarrow 0^+.$ The solution $u$
is unique in a given ball of $\mathcal{G}_{\beta}^{T^*}$ and the
data-solution map $u_0\mapsto u$ into $\mathcal{G}_{\beta}^{T^*}$ is
Lipschitz.
\end{theorem}
\begin{remark}\label{rem1}
\begin{enumerate}
\item[(i)] (Large class of initial data) From the definition of the norm
$\Vert\cdot\Vert_{\mathcal{G}_{\beta}^T}$ and Lemma \ref{LinEstLoc},
if we take $u_0\in L^{(p,\infty)},$ the quantity
$\|G_{\epsilon,\delta}(t)u_0\|_{\mathcal{G}_{\beta}^T}$ is finite.
\item[(ii)] (Regularity) If the initial data  is such that
 \[\sup_{-T<t<T}\vert t\vert^\beta\Vert G_{\epsilon,\delta}(t)u_0\Vert_{(p(\alpha+1),d)}<\infty,\]
  for $1\leq d <\infty,$ then the local mild solution verifies
\[\sup_{-T^*<t<T^*}\vert t\vert^\beta\Vert u\Vert_{(p(\alpha+1),d)}<\infty,\]
(possibly reducing the time of existence $T^*$).
\item[(iii)] (Finite energy solutions) From Theorem \ref{LocalTheo}
some results of local existence in Sobolev spaces can be recovered.
For that, notice
that $H^{s}(\mathbb{R}
^{n})\hookrightarrow L^{(p(\alpha+1),\infty)},$ for
 $s>0$ such that
$2<p(\alpha+1)\leq\frac{2n}{n-2s}$ if $n>2s$
($2<p(\alpha+1)\leq\infty$ if $n<2s$). Therefore, if $u_0\in H^s,$
then
\begin{align*}
\Vert G_{\epsilon,\delta}(t)u_0\Vert_{\mathcal{G}_{\beta}^T}&\leq C\sup_{-T<t<T}|t|^\beta\Vert G_{\epsilon,\delta}(t)u_0\Vert_{H^s}\\
&\leq C\sup_{-T<t<T}|t|^\beta
\Vert u_0\Vert_{H^s}<\infty.
\end{align*}
Consequently, Theorem \ref{LocalTheo} guarantee the existence of a
mild solution $u:(-T^*,T^*)\rightarrow
L^{(p(\alpha+1))}(\mathbb{R}^n)$ in $\mathcal{G}_{\beta}^{T^*}.$ On
the other hand, for the same initial data $u_0\in
H^s(\mathbb{R}^n),$ suppose $v\in C([-T_0,T_0];H^s(\mathbb{R}^n))$
the unique energy finite solution
 for some $T_0$ small enough. By the embedding
$H^s\hookrightarrow L^{(p(\alpha+1),\infty)},$ we obtain that $v\in
\mathcal{G}_{\beta}^{T_0}.$ Thus, taking $T_0$ small enough, the
uniqueness of solution given in Theorem \ref{LocalTheo}, implies
that $u=v$ on $[-T_0,T_0]$ and consequently, $u\in
C([-T_0,T_0];H^s).$
\end{enumerate}
\end{remark}

Before to give the proof of Theorem \ref{LocalTheo}, we enunciate a result related to the existence of radial solutions. First of all, we
recall that a solution $u$ in ${\mathcal{G}_{\beta}^T}$ is said to be radially symmetric, or simply radial, for a.e.
$0<\left\vert t\right\vert <T$, if $u(Rx,t)=u(x,t)$ a.e. $x\in\mathbb{R}^{n}$
for all $n\times n$-orthogonal matrix $R$. Then, we have the following corollary.

\begin{corollary}
\label{cor1}
Under the hypotheses of Theorem \ref{LocalTheo},
if the initial data $u_{0}$ is
radially symmetric, then the corresponding solution $u$ is radially symmetric for a.e.
$0<\left\vert t\right\vert <T$.
\end{corollary}

\begin{proof} (Proof of Theorem \ref{LocalTheo}) The proof of Theorem \ref{LocalTheo} will be obtained as an
application of the Banach fixed point theorem. First notice that, by
hypothesis on the initial data, we have that
\begin{eqnarray*}
\|G_{\epsilon,\delta}(t)u_0\|_{\mathcal{G}_{\beta}^T}:=\sup_{-T<t<T}|t|^{\beta}\|G_{\epsilon,\delta}(t)u_0\|_{(p(\alpha+1),\infty)}\equiv\frac{K}{2}<\infty.
\end{eqnarray*}
We consider the mapping $\Upsilon$ defined by
\begin{eqnarray}
\Upsilon(u(t))=G_{\epsilon,\delta}(t)u_0+i\int_{0}^t
G_{\epsilon,\delta}(t-\tau)f(|u(x,\tau)|)u(x,s)d\tau.
\end{eqnarray}
Then, we will prove that $\Upsilon$ defines a contraction on
 $(B_{K},d)$ where $B_{K}$ denotes the closed ball $\{u\in \mathcal{G}_{\beta}^{T^*}: \|u\|_{\mathcal{G}_{\beta}^{T^*}}\leq
 K\}$ endowed with the complete metric $d(u,v)=
 \|u-v\|_{\mathcal{G}_{\beta}^{T^*}}$ for some $0<T^*\leq T.$
In fact, let us consider $0<T^*\leq T$ such that
$\tilde{C}K^{\alpha}{(T^*)}^{1-\beta(\alpha+1)}<\frac{1}{2}$ where $\tilde{C}$
denotes the constant $C_3$ or $C_4$ in Lemma \ref{NonEstLoc}. Then,
from Lemma \ref{NonEstLoc} with $v=0$ we get
\begin{align*}
\Vert \Upsilon(u)\Vert_{\mathcal{G}_{\beta}^{T^*}} &\leq
\|G_{\epsilon,\delta}(t)u_0\|_{\mathcal{G}_{\beta}^{T^*}}+\|\mathcal{F}(u)\|_{\mathcal{G}_{\beta}^{T^*}}\leq \frac{K}{2}+ \tilde{C}K^{\alpha+1}({T^*})^{1-\beta(\alpha+1)}\\
&\leq  \frac{K}{2}+ \frac{K}{2}=K,
\end{align*}
for all $u\in B_{K}.$ Consequently, $ \Upsilon(B_K)\subset B_K.$ Now
assuming that $u,v\in B_{K},$ from Lemma \ref{NonEstLoc} we obtain
\begin{eqnarray}
\Vert
\Upsilon(u(t))-\Upsilon(v(t))\Vert_{\mathcal{G}_{\beta}^{T^*}}&=&\Vert
\mathcal{F}(u)-\mathcal{F}(v)\Vert_{\mathcal{G}_{\beta}^{T^*}}\nonumber\\
&\leq & 2\tilde{C} K^\alpha ({T^*})^{1-\beta(\alpha+1)}\Vert
u-v\Vert_{\mathcal{G}_{\beta}^{T^*}}.\label{lo1}
\end{eqnarray}
Thus, as
$\tilde{C}K^{\alpha}({T^*})^{1-\beta(\alpha+1)}<\frac{1}{2}$ the map
$\Upsilon$ is a contraction on $(B_K,d).$ Thus, the Banach fixed
point theorem implies the existence of a unique solution $u\in
\mathcal{G}_{\beta}^{T^*}.$ Through standard argument one can prove
that $u(t)\rightarrow u_0$ as $t\rightarrow0,$ in the sense of
distributions \cite{LucEldPab}. On the other hand, in order to prove the local
Lipschitz continuity of the data-solution map, we consider $u,v$ two
local mild solutions with initial data $u_0, v_0,$ respectively.
Then, as in estimate (\ref{lo1}) we get
\begin{eqnarray}
\Vert u-v\Vert_{\mathcal{G}_{\beta}^{T^*}} &= & \Vert
G_{\epsilon,\delta}(t)(u_0-v_0)\Vert_{\mathcal{G}_{\beta}^{T^*}}+\Vert
\mathcal{F}(u)-\mathcal{F}(v)\Vert_{\mathcal{G}_{\beta}^{T^*}}\nonumber\\
&\leq & \Vert
G_{\epsilon,\delta}(t)(u_0-v_0)\Vert_{\mathcal{G}_{\beta}^{T^*}}+ 2\tilde{C}
K^\alpha ({T^*})^{1-\beta(\alpha+1)}\Vert
u-v\Vert_{\mathcal{G}_{\beta}^{T^*}}.\nonumber
\end{eqnarray}
Since $2\tilde{C} K^\alpha ({T^*})^{1-\beta(\alpha+1)}<1,$ from last inequality
the local Lipschitz continuity of the data-solution map holds.
\end{proof}

\textbf{Proof of Corollary \ref{cor1}}  From the fixed point argument used in the
proof of Theorem \ref{LocalTheo}, we can see that the local solution $u$ as the
limit in $\mathcal{G}^T_{\beta}$ of the Picard
sequence
\begin{eqnarray}
 u_{1}   =G_{\epsilon, \delta}(t)(u_{0}){,}\  \  \ u_{k+1} =u_{1}+\mathcal{F}(
u_{k}),\  k\in  \mathbb{N}.\text{ } \label{sequencee}
\end{eqnarray}
Since the symbol of the group
$G_{\epsilon,\delta}(t)$ is radially symmetric for each
fixed $0<t<T,$ it follows that $G_{\epsilon,\delta}(t)u_{0}$ is radial, provided that $u_{0}$ is radial. Furthermore,
since the nonlinear term $\mathcal{F}(u)$ is radial when
$u$ are radial, an induction argument gives that the sequence $\{u_{k}\}_{k\in \mathbb{N}}%
$ given in (\ref{sequencee}) is radial. Since pointwise convergence
preserves radial symmetry, and $\mathcal{G}_{\beta}^{T}$ implies (up
to a subsequence) almost everywhere pointwise convergence in the
variable $x,$ for a.e. fixed $t\neq0$, it follows that $u(x,t)$ is
radially symmetric. $\Box$

\subsection{Global-in-time solutions}
\begin{theorem}[Global-in-time solutions]
\label{GlobalTheo} Let $1\leq \alpha<\infty $  and  assume that
$(\alpha+1)\sigma<1.$ Consider either
$\frac{n\alpha}{4(\alpha+2)}<1$ if $A=\Delta^2,$ or
$\frac{(2n-d)\alpha}{4(\alpha+2)}<1$ if
$A=\sum_{i=1}^d\partial_{x_ix_ix_ix_i}.$ Suppose further that  $
\xi>0$ and $M>0$ satisfy the inequality
$\xi+\widetilde{C}M^{\alpha+1}\leq M$ where
$\widetilde{C}=\widetilde{C}(\alpha, n)$ is the constant $C_1$ or $C_2$ in Lemma
\ref{EstNonGlo}. If $u_0\in \mathcal{D}_\sigma,$ with
$\sup_{t>0}t^\sigma\Vert
G_{\epsilon,\delta}(t)u_0\Vert_{(\alpha+2,\infty)}<\xi,$ then
the initial value problem (\ref{FoSch}) has a unique global-in-time
mild solution $u\in \mathcal{G}^{\infty}_{\sigma}$ with $\Vert
u\Vert_{\mathcal{G}^{\infty}_\sigma}\leq M,$ such that
 $\lim_{t \rightarrow 0} u(t)=u_0$ in distribution sense. Moreover, if $u,v$ are two
global mild solutions with respective initial data $u_0,v_0,$ then
\begin{eqnarray}
\Vert u-v\Vert_{\mathcal{G}^{\infty}_\sigma}\leq C \Vert
G_{\epsilon,\delta}(t)(u_0-v_0)\Vert_{\mathcal{G}^{\infty}_\sigma}.
\end{eqnarray}
Additionally, if $G_{\epsilon,\delta}(t)(u_0-v_0)$ verifies the
stronger decay
\[\sup_{t>0}\vert t\vert^\sigma(1+\vert t\vert)^\varsigma\Vert G_{\epsilon}(t)(u_0-v_0)\Vert_{(\alpha+2,\infty )}<\infty,\]
 for some $\varsigma>0$ such that $\sigma(\alpha+1)+\varsigma<1,$ then
\begin{align}\label{stronger}
\sup_{t>0}\vert t\vert^\sigma(1+\vert t\vert)^\varsigma&\Vert u(t)-v(t)\Vert_{(\alpha+2,\infty )}\leq C\sup_{t>0}\vert t\vert^\sigma(1+\vert t\vert)^\varsigma \Vert G_{\epsilon}(t)(u_0-v_0)\Vert_{(\alpha+2,\infty )}.
\end{align}
\end{theorem}
\begin{remark}
\begin{enumerate}
\item[(i)] (Regularity) In addition to the assumptions of Theorem \ref{GlobalTheo}, if we consider that the
initial data verifies
\[\sup_{-\infty<t<\infty}t^\sigma \Vert G_{0,\delta}(t)u_0\Vert_{(\alpha+2,d)}<\infty\]
for some $1\leq d <\infty,$ then there exists $\xi_0$ such that if
\[\sup_{-\infty<t<\infty}t^\sigma \Vert G_{0,\delta}(t)u_0\Vert_{(\alpha+2,d)}\leq \xi_0,\]
then global
solution provided in Theorem \ref{GlobalTheo} satisfies that
\[\sup_{-\infty<t<\infty}t^\sigma \Vert u(t)\Vert_{(\alpha+2,d)}< \infty.\]
\item[(ii)] (Radial solutions) As in Corollary \ref{cor1}, if the initial data $u_0$ is radially symmetric, then
the global-in-time solution $u$ is radially symmetric for a.e.
$t\neq0.$
\item[(iii)] (Asymptotic stability) Following the proof of (\ref{stronger}) we can obtain that if $u,v$ are global mild solutions of  the Cauchy problem (\ref{FoSch}) given by Theorem \ref{GlobalTheo}, with initial data $u_0,v_0\in \mathcal{D}_\sigma$ respectively, satisfying
\[\lim_{t \rightarrow \infty} t^\sigma(1+ t)^\varsigma \Vert
G_{\epsilon}(t)(u_0-v_0)\Vert_{(\alpha+2,\infty )}=0,\]
then $\lim_{t \rightarrow \infty}t^\sigma(1+ t)^\varsigma \Vert
u(t)-v(t)\Vert_{(\alpha+2,\infty )}=0.$
\item[(iv)] (Biharmonic and anisotropic biharmonic global solutions) Theorem
\ref{GlobalTheo} gives existence of global mild solution for Cauchy
problem associated to equation (\ref{FoNLS}) in the class $
\mathcal{G}^{\infty}_{\sigma}$. The proof was based on the time-decay estimate of the group $G_{0,\delta}(t)$ given in Lemma
\ref{LinEstGlo}. However, taking into account that if $\epsilon=0$
the time-decay estimate in Lemma \ref{LinEstLoc} holds true for all
$t\neq 0,$ we are able to prove the existence of global-in-time mild
solutions for the Cauchy problem associated to equation
(\ref{FoNLS}) in the class $G^\infty_{\sigma(p)}$ defined as the set
of Bochner measurable functions $u:(-\infty ,\infty )\rightarrow
L^{(p(\alpha+1),\infty )}$ such that
\[
\| u\|_{\mathcal{G}^{\infty}_{\sigma(p)}}=\sup_{-\infty
<t<\infty}|t|^{\sigma(p)}\| u(t)\| _{(p(\alpha+1),\infty )}<\infty,
\]
where $\sigma(p)$ is  given by
\begin{equation}\label{Defsigmabb}
\sigma(p)=\left\{
\begin{aligned}
&\frac{1}{\alpha}-\frac{n}{4p(\alpha+1)},\ \mbox{if}\ A=\Delta^2,\\
&\frac{1}{\alpha}-\frac{2n-d}{4p(\alpha+1)},\ \mbox{if}\
A=\sum_{i=1}^d\partial_{x_ix_ix_ix_i}.
\end{aligned}
\right.
\end{equation}
Here $p,\alpha$ must verify $1\leq \alpha<\infty,$
$(1/p,1/{(\alpha+1)p})\in\Xi_0\setminus\partial\Xi_0$ and
$\frac{4p}{n\alpha}<1<\frac{4p(\alpha+1)}{n\alpha}$ if $A=\Delta^2$
or, $\frac{4p}{(2n-d)\alpha}<1<\frac{4p(\alpha+1)}{(2n-d)\alpha}$ if
$A=\sum_{i=1}^d\partial_{x_ix_ix_ix_i.}$
\end{enumerate}
\end{remark}
\begin{corollary}\label{self}(Biharmonic and anisotropic biharmonic self-similar solutions).
Let $\epsilon=0,$ $1\leq \alpha<\infty $  and  assume that
$(\alpha+1)\sigma<1.$ Consider either
$\frac{n\alpha}{4(\alpha+2)}<1$  if $A=\Delta^2,$ or
$\frac{(2n-d)\alpha}{4(\alpha+2)}<1$ if
$A=\sum_{i=1}^d\partial_{x_ix_ix_ix_i}.$ Assume that the initial data $u_0$ is  a homogeneous function of degree $\frac{-4}{\alpha}.$ Then the solution $u(t, x)$ provided by Theorem \ref{GlobalTheo} is self-similar, that is, $u(t, x) = \lambda^{\frac{4}{\alpha}}
 u(\lambda^4 t, \lambda x)$ for all $\lambda > 0$, almost everywhere for $x\in \mathbb{R}^n$ and $t>0.$
\end{corollary}

\begin{remark}
 An admissible class of initial data for the existence of self-similar solutions in Corollary \ref{self} is given by
 the set of functions $u_0(x)=P_m(x)\vert x\vert^{-m-\frac{4}{\alpha}}$ where $P_m(x)$ is a homogeneous polynomial of degree $m$.
\end{remark}

\begin{proof} The proof of Theorem \ref{GlobalTheo} will be also obtained
as an application of the Banach fixed point Theorem. We denote by $B_{M}$
the set of $u\in\mathcal{G}^{\infty}_{\sigma}$ such that
\[\Vert u \Vert_{\mathcal{G}^\infty_\sigma}\equiv\sup_{-\infty<t<\infty}|t|^\sigma\| u(t)\|_{(\alpha+2,\infty )}\leq M,\]
endowed with the complete metric $d(u,v)=\sup_{-\infty<t<\infty}|t|^\sigma\Vert u(t)-v(t)\Vert_{(\alpha+2,\infty )}.$ We
will show that the mapping $\Upsilon$ defined by
\begin{eqnarray}
\Upsilon(u(t))=G_{\epsilon,\delta}(t)u_0+i\int_{0}^t
G_{\epsilon,\delta}(t-\tau)f(|u(x,\tau)|)u(x,s)d\tau,
\end{eqnarray}
is a contraction on $(B_{M},d).$ From the assumptions on the initial
data and Lemma \ref{EstNonGlo} (with $v=0$), we have (for all $u\in B_M$)
\begin{align}
\Vert \Upsilon(u)\Vert_{\mathcal{G}^\infty_\sigma} & \leq  \Vert G_{\epsilon,\delta}(t)u_0\Vert_{\mathcal{G}^\infty_\sigma}+\Vert \mathcal{F}(u)\Vert_{\mathcal{G}^\infty_\sigma} \leq  \xi + \tilde{C}\Vert u\Vert^{\alpha+1}_{\mathcal{G}^\infty_\sigma}\nonumber\\
& \leq  \xi +\tilde{C} M^{\alpha+1}\leq M,
\end{align}
because $M$ and $\xi$ verify $\xi+\tilde{C} M^{\alpha+1}\leq M$.
Thus, $\Upsilon$ maps $B_M$ itself. On the other hand, Lemma \ref{EstNonGlo}, we get
\begin{equation}
\Vert \Upsilon(u)-\Upsilon(v)\Vert_{\mathcal{G}^\infty_\sigma} \leq \Vert \mathcal{F}(u)-\mathcal{F}(v)\Vert_{\mathcal{G}^\infty_{\sigma}} \leq 2\tilde{C}M^\alpha \Vert u-v\Vert_{\mathcal{G}^\infty_{\sigma}}.\label{f1}
\end{equation}
Since $\tilde{C}M^{\alpha}<1,$ it follows that $\Upsilon$ is a contraction
on $(B_{M},d)$  and consequently, the Banach fixed point theorem implies the existence of a unique global solution $u\in \mathcal{G}^\infty_\sigma.$ In order to prove the continuous dependence of the mild
solutions with respect to the initial data, it suffices to observe
that (\ref{f1}) implies that
\begin{equation*}
\Vert u-v\Vert_{\mathcal{G}^\infty_{\sigma}}\leq \Vert G_{\epsilon,\delta}(t)u_0-G_{\epsilon,\delta}(t)v_0\Vert_{\mathcal{G}_{\sigma}}+CM^\alpha \Vert u-v\Vert_{\mathcal{G}^\infty_{\sigma}}.
\end{equation*}
Thus, as $\tilde{C}M^\alpha<1,$ then $\Vert u-v\Vert_{\mathcal{G}^\infty_{\sigma}}\leq C\Vert
G_{\epsilon,\delta}(t)u_0-G_{\epsilon,\delta}(t)v_0\Vert_{\mathcal{G}^\infty_{\sigma}}.$
Finally, in order to prove the stronger decay, notice that
\begin{align}
 t^\sigma(1+ t)^\varsigma\Vert u(t)-v(t)\Vert_{(\alpha+2,\infty )} & \leq  C\sup_{t>0} t^\sigma(1+ t)^\varsigma \Vert G_{\epsilon,\delta}(t)(u_0-v_0)\Vert_{(\alpha+2,\infty )}\nonumber\\
&+  t^\sigma(1+ t)^\varsigma \Vert \mathcal{F}(u)-\mathcal{F}(v)\Vert_{(\alpha+2,\infty )}.\label{stn1}
\end{align}
Since $\Vert u \Vert_{\mathcal{G}^\infty_\sigma}, \Vert v
\Vert_{\mathcal{G}^\infty_\sigma}\leq M,$ using the change of
variable $\tau\mapsto \tau t$ and noting that
$(1+t)^\varsigma(1+t\tau)^{-\varsigma}\leq t^\varsigma
(t\tau)^{-\varsigma}$ for $\tau\in [0,1],$ we obtain
\begin{align}
& t^\sigma(1+ t)^\varsigma\Vert \mathcal{F}(u)-\mathcal{F}(v)\Vert_{(\alpha+2,\infty )}\leq
t^\sigma(1+t)^\varsigma\int_0^t(t-\tau)^{-\frac{n\alpha}{4(\alpha+2)}}\tau^{-\sigma(\alpha+1)}(1+\tau)^\varsigma\nonumber\\
&\times (\tau^\sigma(1+\tau)^\varsigma\Vert u(\tau)-v(\tau)\Vert_{(\alpha+2,\infty )})\left[ \tau^\sigma\Vert u(\tau)\Vert_{(\alpha+2,\infty )}^{\alpha}+\tau^\sigma\Vert v(\tau) \Vert_{(\alpha+2,\infty)}^{\alpha}\right] ds\nonumber\\
&\leq 2M^\alpha\int_0^1 (1-\tau)^{-\frac{n\alpha}{4(\alpha+2)}}\tau^{-\sigma(\alpha+1)} (1+t)^\varsigma  (1+ t\tau)^{-\varsigma}
((t\tau)^\sigma(1+(t\tau))^\varsigma\nonumber\\
&\times \Vert u(t\tau)-v(t\tau)\Vert_{(\alpha+2,\infty )})ds\nonumber\\
&\leq 2M^\alpha\int_0^1(1-\tau)^{-\frac{n\alpha}{4(\alpha+2)}}\tau^{-\sigma(\alpha+1)} \tau^{-\varsigma} ((t\tau)^\sigma(1+(t\tau))^\varsigma\Vert u(t\tau)-v(t\tau)\Vert _{(\alpha+2,\infty )})d\tau.\label{st2}
\end{align}
Therefore, by denoting $A=\sup_{t>0}t^\sigma(1+ t)^\varsigma\Vert
u(t)-v(t)\Vert_{(\alpha+2,\infty )},$ from (\ref{stn1}) and
(\ref{st2}) we get
\begin{align*}
A\leq C&\sup_{t>0} t^\sigma(1+ t)^\varsigma \Vert G_{\epsilon,\delta}(t)(u_0-v_0)\Vert_{(\alpha+2,\infty )}\\
&+\left (2M^\alpha\int_0^1(1-\tau)^{-\frac{n\alpha}{4(\alpha+2)}}\tau^{-\sigma(\alpha+1)} \tau^{-\varsigma} d\tau\right) A.
\end{align*}
Choosing $M$ small enough such that $2M^\alpha\int_0^1
(1-\tau)^{-\frac{n\alpha}{4(\alpha+2)}}\tau^{-\sigma(\alpha+1)}
\tau^{-\varsigma} d\tau<1,$ we conclude the proof.
\end{proof}

\textbf{Proof of Corollary \ref{self}} We recall that due the fixed point argument used in the
proof of Theorem \ref{GlobalTheo}, the solution $u$ is the
limit in $\mathcal{G}_{\sigma}^{\infty}$ of the Picard
sequence
\begin{equation}
u_{1}  =G_{0,\delta}(t)u_{0}{,}\ \ u_{k+1} =u_{1}+\mathcal{F}(u_{k}),\  k\in \mathbb{N}.\text{ } \label{sequencee1}
\end{equation}
Notice that the initial data $u_0$ satisfying $
 u_{0}(\lambda x)=\lambda
^{-\frac{4}{\alpha}}u_{0}(x) $ belongs to the class
$\mathcal{D}_\sigma$ (see Corollary 2.6 in \cite{LucEldPab}).
Since $\epsilon=0,$ we can obtain
\begin{equation}
u_{1}(\lambda x,\lambda^{4}t)=\lambda^{-\frac{4}{\alpha}}u_{1}(x,t) \label{aux-scal2}
\end{equation}
and then $u_{1}$ is invariant by the scaling
\begin{equation}\label{sc}
u(x,t)\rightarrow u_\lambda(x,t):=\lambda^{\frac{4}{\lambda}}u(\lambda x,\lambda^4 t),\ \lambda>0.
\end{equation}
 Moreover, the nonlinear term
$\mathcal{F}(u)$ is invariant by scaling (\ref{sc}) when $u$ is
also. Therefore, we can employ an induction argument in order to
obtain that all elements $u_{k}$ have the scaling invariance
property (\ref{sc}). Because the norm of $\mathcal{G}_{\alpha
}^{\infty}$ is scaling invariant, we get that the limit $u$ also is
invariant by the scaling transformation $u\rightarrow u_\lambda$, as
required.\ \hfill$\square $\vspace{6pt}\medskip

%%%%%%%%%%%%%%%%%%%%%%%%%%%%%%%%%  Limiting Result %%%%%%%%%%%%%%%%%%%%%%%%%%%%%%%%%%
\section{Vanishing dispersion limit}
This section is devoted to the analysis of the solutions of (\ref{FoSch}) as the second order dispersion
vanishes. More exactly, we study the convergence, $\epsilon\rightarrow 0,$ of the solutions 
of the Cauchy problem
\begin{equation}\label{FoSchedl}
\left\{
\begin{array}{lc}
i\partial _{t}u+\epsilon \Delta u+\delta A u+\lambda |u|^\alpha u=0, & x\in \mathbb{R}^{n},\ \ t\in \mathbb{R}, \\
u(x,0)=u_{0}(x), & x\in \mathbb{R}^{n}, \\
\end{array}
\right.
\end{equation}
to the solutions of
\begin{equation}\label{FoSche=0}
\left\{
\begin{array}{lc}
i\partial _{t}u+\delta A u+\lambda|u|^\alpha u=0, & x\in \mathbb{R}^{n},\ \ t\in \mathbb{R}, \\
u(x,0)=u_{0}(x), & x\in \mathbb{R}^{n}. \\
\end{array}
\right.
\end{equation}
in the framework of  the $H^2(\mathbb{R}^n)$ space.
Throughout this subsection we consider $\alpha$ as a positive even integer. Before to establish our main results, we give some preliminary
facts. First, we recall the following conserved quantities of (\ref{FoSchedl}):
\begin{equation}\label{cc1}
M(u)=\Vert u\Vert^2_{L^2(\mathbb{R}^n)}
\end{equation}
\begin{equation}\label{cc2}
E_{\epsilon,\delta,\lambda}(u)=\delta\|\Delta u\|_{L^2}^2-\epsilon \|\nabla u\|^2_{L^2} +\frac{2\lambda}{\alpha+2}\|u\|^{\alpha+2}_{L^{\alpha+2}}, \ \text{if} \ A=\Delta^2
\end{equation}
\begin{equation}\label{cc3}
E_{\epsilon,\delta,\lambda}(u)=\delta\sum_{i=1}^d\|u_{x_ix_i}\|_{L^2}^2-\epsilon \|\nabla u\|^2_{L^2} +\tfrac{2\lambda}{\alpha+2}\|u\|^{\alpha+2}_{L^{\alpha+2}}, \ \text{if} \ A=\sum_{i=1}^d\partial_{x_ix_ix_ix_i}. \\
\end{equation}
According to the signs of the pair $(\delta,\lambda)$ and the parameter  $\epsilon$ that goes to zero, we have two cases
\begin{enumerate}
\item[$(i)$] Case 1: $\delta \lambda>0$ and $\epsilon\in \mathbb{R}.$
\item[$(ii)$] Case 2:  $\delta \lambda<0$ and $\epsilon\in \mathbb{R}.$
\end{enumerate}
Thus we have the next result.
\begin{proposition}\label{cant} Fix the parameters $\delta=\pm 1,\lambda=\pm 1$ and let $u_\epsilon\in C([-T,T];H^2(\mathbb{R}^n))$ the local solution of
(\ref{FoSchedl}) with initial data $u_0\in H^2(\mathbb{R}^n)$ and $A=\Delta^2$. Then,
\begin{itemize}
\item If $(\epsilon,\delta,\lambda)$ is as in Case 1 or
\item If $(\epsilon,\delta,\lambda)$ is as in Case 2,  $n\alpha <8,$  $\frac{n\alpha}{4(\alpha+2)}\leq 1,$  if $n\neq 2,4,$ and $0\leq \frac{n\alpha}{4(\alpha+2)}<1$ if $n= 2,4.$   
\end{itemize}
Then the following estimate holds
\begin{equation}
\Vert u_\epsilon(t)\Vert_{H^2(\mathbb{R}^n)}\leq C(\|u_0\|_{H^2}, \|u_0\|_{L^{\alpha+2}}).
\end{equation}
\end{proposition}
\begin{proof} First we consider the Case 1. Using the conserved quantities of (\ref{FoSchedl}) given in (\ref{cc1})-(\ref{cc2}), we get
\begin{align}\label{FirIne}
\notag &\Vert u_\epsilon(t)\Vert^2_{L^2}+\Vert \Delta u_\epsilon(t)\Vert^2_{L^2} =M(u_0)+\delta^{-1}E_{\epsilon,\delta,\lambda}(u_0)+\delta^{-1}\epsilon\Vert \nabla
u_\epsilon\Vert^2_{L^2}\\
&-\frac{2\delta^{-1}\lambda}{\alpha+2}\|u_{\epsilon}\|^{\alpha+2}_{L^{\alpha+2}}\leq  M(u_0)+\delta^{-1}E_{\epsilon,\delta,\lambda}(u_0)+\delta^{-1}\epsilon\Vert \nabla u_\epsilon(t)\Vert^2_{L^2}.
\end{align}
At this point we have to consider two subcases. If $\delta^{-1}\epsilon<0,$ taking $0<\vert \epsilon\vert<\frac{1}{2},$ we arrived at
\[\Vert u_\epsilon(t)\Vert^2_{L^2}+\Vert \Delta u_\epsilon(t)\Vert^2_{L^2} \leq  M(u_0)+\delta^{-1}E_{\epsilon,\delta,\lambda}(u_0)\leq  M(u_0)+E_{-\frac{1}{2},1,\delta^{-1}\lambda}(u_0).\]
On the other hand, if  $\delta^{-1}\epsilon>0,$ we  have from (\ref{FirIne})
\begin{align*}
\|u_{\epsilon}(t)\|^2_{H^2}&\leq C(\Vert u_\epsilon(t)\Vert^2_{L^2}+\Vert \Delta u_\epsilon(t)\Vert^2_{L^2}) \\
&\leq  CM(u_0)+CE_{0,1,\delta^{-1}\lambda}(u_0)+\delta^{-1}\epsilon C\Vert u_\epsilon(t)\Vert^2_{H^2}.
\end{align*}
Again, consider $0<\vert \epsilon\vert<\frac{1}{2C}$ to arrive at
 \[\|u_{\epsilon}(t)\|^2_{H^2}\apprle M(u_0)+E_{0,1,\delta^{-1}\lambda}(u_0).\]
In any subcase we obtain the desired result.\\

Now, we consider the Case 2. Consider the restrictions $n\alpha<8,$ $0\leq \frac{n\alpha}{4(\alpha+2)}\leq 1$ if $n\neq 2,4,$ and $0\leq \frac{n\alpha}{4(\alpha+2)}<1$ if $n= 2,4.$ Thus, by applying the Douglas-Niremberg and Young inequalities we get
\begin{align}
&\notag \Vert u_\epsilon(t)\Vert^2_{L^2}+\Vert \Delta u_\epsilon(t)\Vert^2_{L^2} =M(u_0)+\delta^{-1}E_{\epsilon,\delta,\lambda}(u_0)+\delta^{-1}\epsilon\Vert \nabla
u_\epsilon(t)\Vert^2_{L^2}\\ \notag
&-\frac{2\delta^{-1}\lambda}{\alpha+2}\|u_{\epsilon}(t)\|^{\alpha+2}_{L^{\alpha+2}} \leq  M(u_0)+\delta^{-1}E_{\epsilon,\delta,\lambda}(u_0)+\delta^{-1}\epsilon\Vert \nabla u_\epsilon(t)\Vert^2_{L^2}\\ \notag
&+C_1\|u_\epsilon(t)\|^{\frac{n\alpha}{4}}_{H^2}\|u_\epsilon(t)\|^{\alpha+2-\frac{n\alpha}{4}}_{L^2}=M(u_0)+\delta^{-1}E_{\epsilon,\delta,\lambda}(u_0)+\delta^{-1}\epsilon\Vert \nabla u_\epsilon(t)\Vert^2_{L^2}\\ \notag
&+C_1\|u_\epsilon(t)\|^{\frac{n\alpha}{4}}_{H^2}\|u_0\|^{\alpha+2-\frac{n\alpha}{4}}_{L^2} \leq M(u_0)+\delta^{-1}E_{\epsilon,\delta,\lambda}(u_0)+\delta^{-1}\epsilon\Vert \nabla u_\epsilon(t)\Vert^2_{L^2}\\ \label{EstNor1}
&+C_1\mu_0 \|u_\epsilon(t)\|^2_{H^2} +C(\mu_0) \|u_0\|^\kappa_{L^2}
\end{align}
with $\kappa=\frac{8(\alpha+2)-8n\alpha}{8-n\alpha}.$ Taking $0<\mu_0 <\frac{1}{2C_1},$ we obtain from (\ref{EstNor1}) that
\begin{equation}\label{EstNor2}
\|u_\epsilon(t)\|^2_{H^2}\apprle  M(u_0)+\delta^{-1}E_{\epsilon,\delta,\lambda}(u_0)+\delta^{-1}\epsilon\Vert \nabla u_\epsilon(t)\Vert^2_{L^2}+C(\|u_0\|_{L^2}).
\end{equation}
Again, we have two subcases. If $\delta^{-1}\epsilon<0,$ it is easy to see that for $0<|\epsilon|<\frac{1}{2}$
\begin{equation}\label{EstNor3}
\|u_\epsilon(t)\|^2_{H^2}\apprle  M(u_0)+E_{-\frac{1}{2},1, \delta^{-1}\lambda}(u_0)+C(\mu_0, \|u_0\|_{L^2}).
\end{equation}
Finally, if $\delta^{-1}\epsilon>0,$ we use that $\delta^{-1}\epsilon\Vert \nabla u_\epsilon(t)\Vert^2_{L^2}\leq \frac{1}{2}\Vert  u_\epsilon(t)\Vert^2_{H^2}$ for $0<|\epsilon|<\frac{1}{2}$ in (\ref{EstNor2}) to obtain again inequality (\ref{EstNor3}).
\end{proof}

Now, we are in position to establish our main results of this section
\begin{theorem}\label{TheCon1} Consider  $u_{\epsilon}$ and $u$  in the class $C([-T,T];H^2(\mathbb{R}^n))$, the solutions of (\ref{FoSchedl}) and (\ref{FoSche=0}) respectively,  with common initial data  $u_0\in H^2(\mathbb{R}^n)$ and $A=\Delta^2.$ Here  $[-T,T]$ is the common interval of local existence for  $u_{\epsilon}$ and $u$. Suppose $n<4,$ if $\delta \lambda<0$ assume that  $n\alpha <8,$  $\frac{n\alpha}{4(\alpha+2)}\leq 1,$  if $n\neq 2,$ and $0\leq \frac{n\alpha}{4(\alpha+2)}<1$ if $n= 2.$   Then,
\[\lim_{\epsilon\rightarrow 0} \|u_{\epsilon}(t)-u(t)\|_{H^2}=0,\]
for all $t\in [-T,T].$
\end{theorem}
\begin{remark}(Anisotropic dispersion). A version of Theorem \ref{TheCon1} for the anisotropic case, i.e.,
$A=\sum_{i=1}^d\partial_{x_ix_ix_ix_i},$ by replacing the norm
convergence in $H^2$ by the natural norm
$H^(\mathbb{R}^d)H^2(\mathbb{R}^{n-d}),$ is not clear. In fact, we
are not able to bound $\Vert \nabla u_\epsilon\Vert_{L^2}$  or $ \Vert u_\epsilon\Vert^2_{H^1}+\sum_{i=1}^d\Vert  u_{\epsilon_{x_ix_i}}\Vert^2_{L^2}$ in terms of the conserved quantities associated  to (\ref{FoSchedl}) and  independently of $\epsilon.$ 
\end{remark}
\begin{proof} As usual, the mild  solutions associated to (\ref{FoSche=0})
satisfy the integral equation
\begin{equation}\label{IntEque=0}
u(x,t)=G_{0,\delta}(t)u_0(x)+i\int_0^tG_{0,\delta}(t-\tau)f(|u(x,\tau)|)u(x,\tau)d\tau,
\end{equation}
where $G_{0,\delta}$ is define as $G_{\epsilon, \delta}$ given in (\ref{DefGe}), but
with $\epsilon=0.$
Making the difference between the integral
equations (\ref{IntEqu}) and (\ref{IntEque=0}) we get that
\begin{align*}
&\|u_{\epsilon}(t)-u(t)\|_{H^2}\leq \left\|\left[G_{\epsilon,\delta}(t)-G_{0,\delta}(t)\right]u_0\right\|_{H^2}\\
&+\left\|\int_0^tG_{\epsilon,\delta}(t-\tau)|u_{\epsilon}(\tau)|^\alpha u_{\epsilon}(\tau)d\tau-\int_0^tG_{0,\delta}(t-\tau)|u(\tau)|^\alpha u(\tau)d\tau\right\|_{H^2}\leq \\
&\int_0^t\left\|G_{\epsilon,\delta}(t-\tau)\left[|u_{\epsilon}(\tau)|^\alpha u_{\epsilon}(\tau)-|u(\tau)|^\alpha u(\tau)\right]\right\|_{H^2}d\tau+ \left\|\left[G_{\epsilon,\delta}(t)-G_{0,\delta}(t)\right]u_0\right\|_{H^2}\\
&+\int_0^t\left\|[G_{\epsilon,\delta}(t-\tau)-G_{0,\delta}(t-\tau)]|u(\tau)|^\alpha
u(\tau)\right\|_{H^2}d\tau
\end{align*}
Since $G_{\epsilon,\delta}(t)$ is a unitary group on $H^2,$ from last inequality we obtain
\begin{align}\label{IneDif1}
\notag &\|u_{\epsilon}(t)-u(t)\|_{H^2}\leq \int_0^t \left\|\left[|u_{\epsilon}(\tau)|^\alpha u_{\epsilon}(\tau)-|u(\tau)|^{\alpha}u(\tau)\right]\right\|_{H^2}d\tau+\\
&\notag\left\|\left[G_{\epsilon,\delta}(t)-G_{0,\delta}(t)\right]u_0\right\|_{H^2}+\int_0^t\left\|[G_{\epsilon,\delta}(t-\tau)-G_{0,\delta}(t-\tau)]|u(\tau)|^\alpha u(\tau)\right\|_{H^2}d\tau\\
\notag
&\leq \int_0^t \left\||u_{\epsilon}(\tau)-u(\tau)|(|u_{\epsilon}(\tau)|^{\alpha}+|u(\tau)|^{\alpha})\right\|_{H^2}d\tau+ \left\|\left[G_{\epsilon,\delta}(t)-G_{0,\delta}(t)\right]u_0\right\|_{H^2}\\
&+\int_0^t\left\|[G_{\epsilon,\delta}(t-\tau)-G_{0,\delta}(t-\tau)]|u(\tau)|^\alpha
u(\tau)\right\|_{H^2}d\tau.
\end{align}
Then,  from (\ref{IneDif1}) and Proposition \ref{cant} we have
 \begin{align}
\notag\|u_{\epsilon}(t)-u(t)\|_{H^2}&\leq C\int_0^t \left\|u_{\epsilon}(\tau)-u(\tau)\right\|_{H^2}d\tau+ \left\|\left[G_{\epsilon,\delta}(t)-G_{0,\delta}(t)\right]u_0\right\|_{H^2}\\
&+\int_0^t\left\|[G_{\epsilon,\delta}(t-\tau)-G_{0,\delta}(t-\tau)]|u(\tau)|^\alpha
u(\tau)\right\|_{H^2}d\tau.
\end{align}
From Gronwall's inequality we arrived at
\[\|u_{\epsilon}(t)-u(t)\|_{H^2}\leq \Psi_{\epsilon, \delta}(t)+C\int_0^t\Psi_{\epsilon, \delta}(\tau)e^{C(t-\tau)}d\tau,\]
where
\begin{align*}
\Psi_{\epsilon, \delta}(t)= &\left\|\left[G_{\epsilon,\delta}(t)-G_{0,\delta}(t)\right]u_0\right\|_{H^2}\\
&\ \ \ \ \ +\int_0^t\left\|[G_{\epsilon,\delta}(t-\tau)-G_{0,\delta}(t-\tau)]| u(\tau)|^\alpha u(\tau)\right\|_{H^2}d\tau.
\end{align*}
Notice that being $\alpha$ a positive integer we have
\begin{align*}
\Psi_{\epsilon, \delta}(t)&\leq \|u_0\|_{H^2}+\int_0^t\| |u(\tau)|^\alpha u(\tau)\|_{H^2}d\tau\leq \|u_0\|_{H^2}+\int_0^t\|u(\tau)\|^{\alpha+1}_{H^2}d\tau\\
&\leq  \|u_0\|_{H^2}+t\|u_0\|^{\alpha+1}_{H^2}.
\end{align*}
Thus $|\Psi_{\epsilon, \delta}(\tau)e^{C(t-\tau)}|\apprle e^{C(t-\tau)}.$ Since $e^{C(t-\tau)}\in L^1(0,T),$ in order to obtain our result we just have to show that $\Psi_{\epsilon, \delta}(t)\rightarrow 0$ as $\epsilon\rightarrow 0,$ for any $t\in[0,T].$ First, observe that
\[\left\|\left[G_{\epsilon,\delta}(t)-G_{0,\delta}(t)\right]u_0\right\|^2_{H^2}=\int_{\mathbb{R}^n}\langle \xi\rangle^4|e^{-it\epsilon|\xi|^2}-1|^2|\widehat{u_0}(\xi)|^2d\xi.\]
Since $\langle
\xi\rangle^4|e^{-it\epsilon|\xi|^2}-1|^2|\widehat{u_0}(\xi)|^2\apprle
\langle \xi\rangle^4|\widehat{u_0}(\xi)|^2\in L^1(\mathbb{R}^n)$ and
$\langle
\xi\rangle^4|e^{-it\epsilon|\xi|^2}-1|^2|\widehat{u_0}(\xi)|^2\rightarrow
0,$ as $\epsilon\rightarrow 0,$ a.e. on $\mathbb{R}^n,$ by the
Lebesgue's dominated convergence theorem  we have
\[\lim_{\epsilon\rightarrow 0} \left\|\left[G_{\epsilon,\delta}(t)-G_{0,\delta}(t)\right]u_0\right\|_{H^2}=0.\]
From Proposition \ref{cant} we get
\begin{align*}
\|[G_{\epsilon,\delta}(t-\tau)-G_{0,\delta}(t-\tau)]|u(\tau)|^\alpha u(\tau)\|_{H^2}&\leq \||u(\tau)|^\alpha u(\tau)\|_{H^2}\leq \|u(\tau)\|^{\alpha+1}_{H^2}\\
& \apprle [C(\|u_0\|_{H^2}, \|u_0\|_{L^{\alpha+2}})]^{\alpha+1}.
\end{align*}
Moreover notice that
$\|[G_{\epsilon,\delta}(t-\tau)-G_{0,\delta}(t-\tau)]|u(\tau)|^\alpha
u(\tau)\|_{H^2}\rightarrow 0,$ as $\epsilon\rightarrow 0;$ then we
arrived at
\[\lim_{\epsilon\rightarrow 0}\int_0^t\left\|[G_{\epsilon,\delta}(t-\tau)-G_{0,\delta}(t-\tau)]|u(\tau)|^\alpha u(\tau)\right\|_{H^2}d\tau=0,\]
which finishes the proof of the theorem.
\end{proof}

%%%%%%%%%%%%%%%%%%%%%%%%%%%%%%%% BIBLIOGRAFIA%%%%%%%%%%%%%%%%%%%%%%%%%%%%%%%%%%%%%%%%%%%%%%

\end{document}